\documentclass{amsproc}
\usepackage{graphics}
\usepackage{amsmath}
\usepackage{amsfonts}
\usepackage{amssymb}
\usepackage{setspace}
\usepackage{amsthm}
\usepackage{epstopdf}

\newtheorem{thm}{Theorem}[section]

\newtheorem{lem}[thm]{Lemma}

\newtheorem{ex}[thm]{Example}

\newcommand{\ds}{\displaystyle}

\begin{document}

\title{On the range of composition operators on spaces of entire functions}

\author[Saikat Mukherjee]{S. Mukherjee}
\address{Department of Mathematics\\
University of Wyoming\\
Laramie, WY 82071-3036}
\email{smukherj@uwyo.edu}

\author[Farhad Jafari]{F. Jafari}
\address{Department of Mathematics\\
University of Wyoming\\
Laramie, WY 82071-3036}
\email{fjafari@uwyo.edu}

\author[John McInroy]{J. E. McInroy}
\address{Department of Electrical and Computer Engineering\\
University of Wyoming\\
Laramie, WY 82071-3036}
\email{mcinroy@uwyo.edu}

\thanks{The work of second and third authors was partially supported by a grant from AFOSR}

\begin{abstract}
The celebrated Paley-Wiener theorem naturally identifies the spaces of bandlimited functions with subspaces of entire functions of exponential type. Recently, it has been shown that these spaces remain invariant only under composition with affine maps. After some motivation demonstrating the importance of characterization of range spaces arising from the action of more general composition operators on the spaces of bandlimited functions,  in this paper we identify the subspaces of $ L^2 (\mathbb{R}) $ generated by these actions. Extension of these theorems where Paley-Wiener spaces are replaced by the deBranges-Rovnyak spaces are given.
\end{abstract}

\maketitle

\section{Introduction}
Let $M_{f}(r)$ denote the maximum modulus of $f(z)$ for $|z|=r$, i.e., $M_{f}(r) = \ds \max_{|z|=r}|f(z)|$. Recall that an \textit{entire function} $f$ \textit{is of order} $\rho$ if
\begin{center}
$\ds \limsup_{r\rightarrow \infty}\frac{\log \log M_{f}(r)}{\log r}=\rho$, \qquad $0 \leq \rho \leq \infty.$
\end{center}
If the order of an entire function $f$ is finite, we will define another number associated to $f$, called the \textit{type of the function $f$}, which more precisely describes the rate of growth of $f$. The \textit{entire function $f(z)$ of positive order $\rho$ is of type $\sigma$} if
\begin{center}
$\ds \limsup_{r\rightarrow \infty} \frac{\log M_{f}(r)}{r^{\rho}}=\sigma$, \qquad $0 \leq \sigma \leq \infty.$
\end{center}
An entire function $f$ is said to be of \textit{exponential type $\sigma$}, if it is of order $\rho = 1$ and type $\sigma$ with $0 < \sigma < \infty$.

Suppose $f$ is an analytic function in $\Omega \subseteq \mathbb{C}$. $f$ is called \textit{bounded type} in $\Omega$ if $f(z) = \frac{p(z)}{q(z)}$, where both $p$, $q$ are analytic and bounded in $\Omega$ and $q$ is not identically zero. By a theorem of M. G. Kre\v{\i}n, an entire function is of exponential type if it is bounded type in the upper and lower half of the complex plane. In that case, the exponential type of the function is equal to the maximum of its mean types in the upper and lower half planes. Two well-known formulas for \textit{mean type} $h$ of a function $f$ in the upper half plane are
\begin{center}
$h = \ds \limsup_{y\rightarrow \infty} \frac{\log |f(iy)|}{y}$,\\
and $h = \ds \lim_{r\rightarrow \infty} \frac{2}{\pi r}\int_0^{\pi}\log |f(re^{i\theta})|\sin \theta d\theta$.
\end{center}
Mean type is a generalization of exponential type to functions which are not necessarily entire.

Let $H(G)$ be the set of holomorphic functions on a domain $G\subset\mathbb{C}$ and let $X\subseteq H(G)$ be a Banach space of holomorphic functions on $G$. Assume that the embedding $X\rightarrow H(G)$ is continuous with respect to the respective topologies. Suppose $\varphi$ is a holomorphic function from $G$ into $G$. Let $T:X\rightarrow X$ be a bounded operator. We say $T$ is a \textit{composition operator} on $X$ induced by $\varphi$ if $(Tf)(z)=f(\varphi(z))$, for every $f\in X$. In this case, it is natural to denote $T$ by $C_{\varphi}$. We say $T$ is a \textit{weighted composition operator} if $(Tf)(z)=m(z)f(\varphi(z))$, i.e., $T=M_{m}C_{\varphi}$, where $M_{m}$ is a multiplication operator with multiplier $m\in H(G)$ is a holomorphic function on $G$. If, in addition, $m\in Y\subseteq H(G)$, we say $T$ is a \textit{$Y$-weighted composition operator}.

Composition operators have been the subject matter of study over the last few decades (see \cite{CoMa95} for extensive references). In this paper we study these operators on Paley-Wiener and de Branges-Rovnyak spaces. As the composition operators that keep the Paley-Wiener spaces invariant are quite limited, we consider more general composition operators acting on Paley-Wiener spaces and characterize the range of these actions. This problem arises when bandlimited signals are warped through a nonlinear diffeomorphism of the underlying domain and the ensuing signals are being approximated. Characterization of these range spaces allows use of appropriate basis functions to provide  appropriate approximation of these signals. To establish the notation and a detailed introduction, in Section 2 we will discuss the bounded composition operators on Paley-Wiener spaces and extend this result to weighted composition operators. In Section 3 we consider more general, non-affine, composition operators and ask how large can the range of these operators acting on bandlimited signals get? We provide a full characterization of these range spaces as reproducing kernel Hilbert spaces and give some of their properties. Finally, in Section 4 we generalize these results to de Branges-Rovnyak spaces.  The Paley-Wiener spaces are a special instance of these spaces.

\section{Bounded weighted composition operators on Paley-Wiener Spaces}
A finite energy signal is said to be \textit{bandlimited} if its spectrum vanishes outside of a finite interval of the form $[-a,a]$. The smallest such $a$ is called the \textit{bandwidth} of the signal. Such a signal $f(t)$ can be written in the form
\begin{equation}\label{band}
f(t) = \int_{-a}^{a}\widehat{f}(w)e^{-itw}dw
\end{equation}
with $\widehat{f}\in L^{2}([-a,a])$.
The \textit{Paley-Wiener space} $B^{2}_{a}$ is the space of all entire functions of exponential type less than or equal to $a$ whose restriction to the real line belong to the space $L^{2}(\mathbb{R})$. This is a Hilbert Space with the inner product inducing the norm
\begin{center}
$\|f\|^{2}=\ds \int_{-\infty}^{+\infty}|f(x)|^{2}dx$.
\end{center}
The inclusion map from $ B^2_a $ embeds the Paley-Wiener spaces into $ L^2( \mathbb{R}) $ isometrically. Thus without loss of generality we may identify $B^{2}_{a}$ as a closed subspace of $L^{2}(\mathbb{R})$.
It follows immediately from the Paley-Wiener theorem (\cite{PaWi34}) that the space $B^{2}_{a}$ is the space of all entire functions bandlimited to $[-a,a]$. This is in fact a reproducing kernel Hilbert space with the reproducing kernel (see ~\cite{de68}, ~\cite{Sa97})
\begin{center}
$\ds k_{w}(z)=\frac{\sin a(z-\overline{w})}{\pi (z-\overline{w})}$
\end{center}
The ubiquity of bandlimited signals in applications and the natural correspondence of the space of bandlimited functions and Paley-Wiener spaces gives rise to the very natural question: which composition operators preserve these spaces? Equivalently stated, for what $\varphi$'s does $C_{\varphi}$ send $B^{2}_{a}$ into itself? Very recently, Chac$\acute{\text{o}}$n, et. al. (~\cite{ChChGi07}) provided a complete answer to this questions.

\begin{thm}\label{Chacon}
Let $\varphi :\mathbb{C}\rightarrow \mathbb{C}$ be a nonconstant entire function. The operator $C_{\varphi}$ is bounded on $B^{2}_{a}$ if and only if $\varphi (z)=cz+d$, $z\in \mathbb{C}$, with $0<|c|\leq 1$, and $c\in \mathbb{R}$.
\end{thm}
Since the class of warps preserving Paley-Wiener spaces are limited to translations and dilations, that is composition with affine maps only, it would be interesting to ask which weighted composition operators have the same property. As the Paley-Wiener spaces are Fourier spaces, these weighted composition operators naturally correspond to the smearing and warping of the original bandlimited signals.
That is, suppose $T=M_mC_\varphi$ and $T$ acts on $B^2_a$. We would like to know if $\varphi$ is affine, for which $m$, $Tf \in B^2_A$ for some $A$ and for all $f \in B^2_a$. The following theorem answers this question.

\begin{thm}\label{Chacon-cor-weight}
\begin{enumerate}
  \item[(a)] Let $\varphi :\mathbb{C}\rightarrow \mathbb{C}$ be a nonconstant entire function. Then $C_{\varphi}:B^{2}_{a}\rightarrow B^{2}_{|c|a}$ if and only if $\varphi (z)=cz+d$, $z\in \mathbb{C}$, with $c\in \mathbb{R}\setminus\{0\}$.
  \item[(b)] Let $T=M_mC_\varphi$, with $\varphi (z)=cz+d$ and $c\in \mathbb{R}\setminus\{0\}$,  be a weighted composition operator on $B^2_a$, where $m$ is an entire function. Then $T:B^2_a\rightarrow B^2_A$ for $A=\max{\{\left|r-|c|a|,|s+|c|a\right|\}}$ if and only if $\widehat{m} \in C_o(\mathbb{R})$ with $supp(\widehat{m})\subseteq [r,s]$.
\end{enumerate}
\end{thm}

\begin{proof}
   (a) Firstly,   note that if $f(z)$ is an entire function with order $\rho$ and type $\sigma$, then the order and type of $f(cz)$ are $\rho$ and $|c|^{\rho}\sigma$. To see this,
    suppose $f(z)=\ds \sum_{n=0}^{\infty} a_{n}z^{n}$. Then the order and type of $f$ can be determined by the following two formulas: (see ~\cite{Le96})
    \begin{center}
    $\rho=\ds \limsup_{n\rightarrow \infty} \frac{n\log n}{\log(1/|a_{n}|)}$,
    $\sigma=\frac{1}{\rho e} \ds \limsup_{n\rightarrow \infty} (n\sqrt[n]{|a_{n}|^{\rho}})$.
    \end{center}
    Let $\rho_{1}$ and $\sigma_{1}$ be the order and type of $f(cz)$. Then,
    \begin{center}
    $\rho_{1}=\ds \limsup_{n\rightarrow \infty} \frac{n\log n}{\log(1/|a_{n}c^{n}|)}=\rho$,
    \end{center}
    and,
    \begin{center}
    $\sigma_{1}=\frac{1}{\rho e} \ds \limsup_{n\rightarrow \infty} (n\sqrt[n]{|a_{n}c^{n}|^{\rho}})=|c|^{\rho}\sigma$.
    \end{center}
Now, to prove the theorem, suppose
$f\in B^{2}_{a}$, then $f$ is an entire function of exponential type less than or equal to $a$. It is clear from Theorem \ref{Chacon} that the only part we need to consider in this proof is when $\varphi (z) = cz$. Now,
\begin{center}
$\ds \int_{\mathbb{R}}|C_{\varphi}f(x)|^{2}dx = \ds \int_{\mathbb{R}}|f(cx)|^{2}dx=\frac{1}{|c|}\ds \int_{\mathbb{R}}|f(x)|^{2}dx< \infty$.
\end{center}
Also, from above, $C_{\varphi}f$ is of exponential type less than or equal to $|c|a$. Hence $C_{\varphi}f \in B^{2}_{|c|a}$.
The converse is also true because, using P\'{o}lya's theorem (~\cite{ChChGi07}, ~\cite{Po26}), it follows that if $C_{\varphi}(B^{2}_{a}) \subseteq B^{2}_{A}$, for any $A$, then $\varphi$ is affine.

(b) Since $\varphi (z)=cz+d$, it is clear from part(a) that $C_{\varphi}:B^2_a\rightarrow B^2_{|c|a}$. Therefore $supp(\widehat{C_\varphi f}) \subseteq [-|c|a , |c|a]$. Now suppose $\widehat{m} \in C_o(\mathbb{R})$, with $supp(\widehat{m}) \subseteq [r , s]$. Then $supp(\widehat{m}\ast \widehat{C_\varphi f}) \subseteq [r-|c|a , s+|c|a]$. Since $\widehat{m} \in C_o(\mathbb{R})$, $\widehat{m} \in L^1(\mathbb{R})$. Also, since $\widehat{m}$ is defined, $m\in L^1(\mathbb{R})$ and then by the inverse Fourier transform formula,
\begin{equation}
m(x) \;=\; \ds \int_{-\infty}^{\infty} \widehat{m}(t) e^{-ixt}dt,\nonumber
\end{equation}
and this implies,
\begin{equation}
|m(x)| \leq \ds \int_{-\infty}^{\infty} |\widehat{m}(t)|dt \;=\; \|\widehat{m}\|_1 < \infty.\nonumber
\end{equation}
Hence, $m\in L^\infty(\mathbb{R})$, and therefore $Tf\;=\;M_mC_\varphi f \in L^2(\mathbb{R})$. Hence for all $f\in B^2_a$, $Tf \in B^2_A$ for $A=\max{\{|r-|c|a|,|s+|c|a|\}}$.\\
Conversely, suppose $T:B^2_a\rightarrow B^2_A$ for some $A$. If $\widehat{m}$ doesn't have compact support, then for some $f\in B^2_a$, $\widehat{m}\ast \widehat{C_\varphi f}$ will fail to have compact support. This is a contradiction to the fact that $Tf\;=\;M_mC_\varphi f \in B^2_A$ for some $A$.
\end{proof}

\section{Range of composition operators acting on Paley-Wiener spaces}

Theorem \ref{Chacon} demonstrates that Paley-Wiener spaces are rather rigid under diffeomorphisms of the underlying domain. However, nonlinear warps (e.g. chirps) occur in various applications and the bandlimited input signals are transported into subspaces of $ L^2 (\mathbb{R}) $. Characterization of these subspaces would lead to appropriate approximation theorems which allows for proper representation of the output signals. To describe these range spaces, we recall the following well known theorem (see ~\cite{Si76}, for example).

\begin{thm}\label{Singh}
Suppose $(X,\Sigma,\lambda)$ is a $\sigma$-finite measure space and $\varphi$ is a measurable function on $X$ into itself. Then $C_{\varphi}$ is a bounded composition operator on $L^{2}(\lambda)$ if and only if there exists a constant $c>0$ such that $\lambda \varphi ^{-1}(E) \leq c \lambda (E)$ for every measurable set $E$. Here $\varphi ^{-1}(E)$ is the pull-back of the set $E$.
\end{thm}

This is equivalent to saying that $C{_\varphi}:L^{2}(\lambda)\rightarrow L^{2}(\lambda)$ if and only if $\lambda \varphi ^{-1}$ is absolutely continuous with respect to $\lambda$ and the Radon-Nikod\'{y}m derivative of $\lambda \varphi ^{-1}$ with respect to $\lambda$ is essentially bounded.  Clearly this will serve as a necessary condition for the theorems to follow. Before, we provide a much deeper characterization of the range spaces of non-affine warps (or affine warps for that matter), we note an elementary result which arises from Plancherel's theorem.  This observation shows that if the support of the bandlimited functions is loosened, then the range spaces is arbitrarily close to larger Paley-Wiener spaces.  More precisely,

\begin{thm}\label{JaSa1}
Suppose $\varphi :\mathbb{R}\rightarrow \mathbb{R}$ is a measurable function that satisfies the hypothesis of Theorem \ref{Singh}, i.e., there exists a constant $c>0$ such that $m\varphi ^{-1}(E) \leq c m(E)$ for every measurable set $E$, where $m$ is the Lebesgue measure on the Borel subsets of $\mathbb{R}$. Let $f\in B^{2}_{a}$. Then for every $\varepsilon >0$, there exists a positive $N$ such that for every $A \geq N$ there exists an $h\in B^{2}_{A}$ with $\|C_{\varphi}f-h\| < \varepsilon$.
\end{thm}

\begin{proof}
Suppose $\varphi :\mathbb{R}\rightarrow \mathbb{R}$ is a measurable function and there is $c>0$ such that $m\varphi ^{-1}(E) \leq c m(E)$ for every measurable set $E \subseteq \mathbb{R}$, where $m$ is the Lebesgue measure on the Borel subsets of $\mathbb{R}$. Let $f \in B^{2}_{a}$ be an arbitrary function. Then since $f\in L^{2}(\mathbb{R})$, by Theorem \ref{Singh} it follows that $C_{\varphi}f \in L^{2}(\mathbb{R})$. Define,
\begin{center}
$h(z):= \ds \int_{-A}^{A} \widehat{C_{\varphi}f}(t)e^{-izt} dt$
\end{center}
Then it is clear that $h\in B^{2}_{A}$.\\
We will show that, $\widehat{h}(t)=(\chi_{A}\widehat{C_{\varphi}f})(t)$, $\forall t \in \mathbb{R}$, where $\chi_{A}$ is the characteristic function of $[-A,A]$.\\
\begin{eqnarray}
\widehat{h}(t) &=& \ds \int_{-\infty}^{\infty} h(x) e^{ixt}dx\nonumber\\
&=& \ds \int_{-\infty}^{\infty} \left\{\ds \int_{-A}^{A}\widehat{C_{\varphi}f}(s)e^{-ixs}ds\right\} e^{ixt}dx\nonumber\\
&=& \ds \int_{-A}^{A} \widehat{C_{\varphi}f}(s)\left\{\ds \int_{-\infty}^{\infty}e^{ix(t-s)}dx\right\}ds\nonumber\\
&=& \ds \int_{-A}^{A} \widehat{C_{\varphi}f}(s)\delta (t-s)ds\nonumber\\
&=& \ds \left\{\begin{array}{lcr}
 \widehat{C_{\varphi}f}(t) &\text{if}\quad t\in [-A,A]\\
 0 &\text{otherwise}\\
\end{array}\right.\nonumber\end{eqnarray}
Therefore using Plancherel's theorem we have:
\begin{eqnarray}
\|C_{\varphi}f-h\| &=& \|\widehat{C_{\varphi}f}-\widehat{h}\| \nonumber\\
 &=& \|\widehat{C_{\varphi}f}\|_{L^2([-A,A]^c)},\nonumber
\end{eqnarray}
which goes to zero as $A\rightarrow +\infty$. Hence the statement of the theorem is proved.
\end{proof}
As a consequence if $\varphi :\mathbb{C}\rightarrow \mathbb{C}$ is an entire function, whose restriction to the real line is real and satisfies all the hypothesis of Theorem \ref{JaSa1}, then the composition operator induced by $\varphi$ will satisfy the conclusion of the above theorem.
\begin{ex}[\cite{Si76}, \cite{SiMa93}]
Denote $\varphi_{\mathbb{R}}$ as the restriction of $\varphi$ on $\mathbb{R}$. If $\varphi: \mathbb{C} \rightarrow \mathbb{C}$ is an entire function with $\varphi_{\mathbb{R}}:\mathbb{R} \rightarrow \mathbb{R}$ monotone and $\frac{1}{\varphi '_{\mathbb{R}}}$ essentially bounded, then the composition operator induced by $\varphi$ satisfies the conclusion of Theorem \ref{JaSa1}. An example of such a function is $\varphi(z)=az^{3}+bz^{2}+cz+d$, where $a,b,c,d \in \mathbb{R}$ with $b^{2}<3ac$.
\end{ex}

Now we will show that the range spaces are reproducing kernel Hilbert spaces with reproducing kernel generated by the reproducing kernel of $B^{2}_{a}$. These reproducing kernels form a basis for the range spaces and may be used to provide best approximation results for the images.  We know that the inner product of $B^{2}_{a}$ is defined by $\left(f(\cdot) \; , \; g(\cdot)\right)_{B^{2}_{a}}\; = \; \ds \int_{\mathbb{R}} f(t)\overline{g(t)}dt$. Suppose $\mathfrak{F}(\mathbb{C})$ denotes the linear space of all complex valued functions on $\mathbb{C}$. Consider the composition map $C_{\varphi}:B^{2}_{a}\rightarrow \mathfrak{F}(\mathbb{C})$. Let $h_\varphi$ be a mapping from $\mathbb{C}$ into $B^{2}_{a}$ such that $\left(C_{\varphi}f\right)(z)\; = \; \left(f(\cdot) \; , \; h_\varphi(z,\cdot)\right)_{B^{2}_{a}}$ for all $f\in B^{2}_{a}$.

But we know, $\ds f(z) = \left(f(\cdot) \; , \; \frac{\sin a(\cdot - \overline{z})}{\pi(\cdot - \overline{z})}\right)_{B^{2}_{a}}$. This implies, $\ds C_{\varphi}f(z)\; = \; f(\varphi(z))\; = \; \left(f(\cdot) \; , \; \frac{\sin a(\cdot - \overline{\varphi(z)})}{\pi(\cdot - \overline{\varphi(z)})}\right)_{B^{2}_{a}}$. Hence we have, $\ds h_\varphi(z,\cdot)\; = \; \frac{\sin a(\cdot - \overline{\varphi(z)})}{\pi(\cdot - \overline{\varphi(z)})}$. Then by Saitoh's theorem (see ~\cite{Sa97}), the range space $ran(C_{\varphi})$ is a reproducing kernel Hilbert space with the reproducing kernel given by,
\begin{eqnarray}
K^{(\varphi)}(z,w) &=& \left(h_\varphi(w,\cdot) \; , \; h_\varphi(z,\cdot)\right)_{B^{2}_{a}}\nonumber\\
&=& \ds \left(\frac{\sin a(\cdot - \overline{\varphi (w)})}{\pi(\cdot - \overline{\varphi (w)})} \; , \; \frac{\sin a(\cdot - \overline{\varphi (z)})}{\pi(\cdot - \overline{\varphi (z)})}\right)_{B^{2}_{a}}\nonumber\\
&=& \ds \int_{-\infty}^{\infty} \frac{\sin a(t - \overline{\varphi (w)})}{\pi(t - \overline{\varphi (w)})} \frac{\sin a(t - \varphi (z))}{\pi(t - \varphi (w))}dt\nonumber\\
&=& \ds \int_{-a}^{a} e^{-it\overline{\varphi (w)}}e^{it\varphi (z)}dt\nonumber\\
&=& \ds \frac{\sin a(\varphi (z) - \overline{\varphi (w)})}{\pi(\varphi (z) - \overline{\varphi (w)})}.\nonumber\end{eqnarray}

Then we can introduce the inner product in $ran(C_{\varphi})$ in the following manner:\\
It is well-known that the set $\{K^{(\varphi)}(\cdot,n)\}_{n\in \mathbb{Z}}$ forms an orthonormal basis in $ran(C_{\varphi})$. Then for any $F, G \in ran(C_{\varphi})$, there exist $\{a_n\}_{n\in \mathbb{Z}}, \{b_n\}_{n\in \mathbb{Z}}$ in $\mathbb{C}$ such that $\ds F\;=\; \sum_n a_n K^{(\varphi)}(\cdot,n)$ and $\ds G\;=\; \sum_n b_n K^{(\varphi)}(\cdot,n)$. Then the inner product can be written as,
\begin{equation}
\left(F\;,\;G\right)_{ran(C_{\varphi})} \;=\; \left(\sum_n a_n K^{(\varphi)}(\cdot,n)\; , \;\sum_m b_m K^{(\varphi)}(\cdot,m)\right)_{ran(C_{\varphi})}\; = \; \sum_{m,n} \overline{b_m} a_n K^{(\varphi)}(m,n)\nonumber
\end{equation}
Since $C_{\varphi}:B^{2}_{a}\rightarrow ran(C_{\varphi})$ is a linear and bounded operator, by the closed graph theorem, $C_{\varphi}$ is a closed operator.

\begin{ex}
Suppose $\varphi(z)=z^{3}+z$. Then $C_{\varphi}:B^{2}_{a}\rightarrow ran(C_{\varphi})\subset L^{2}(\mathbb{R})$ and $ran(C_{\varphi})$ is a reproducing kernel Hilbert space with reproducing kernel $\ds K^{(\varphi)}(z,w)\;=\; \frac{\sin a(z^{3}+z-\overline{w}^{3}-\overline{w})}{\pi(z^{3}+z-\overline{w}^{3}-\overline{w})}$.
\end{ex}

The following theorem will provide the characterization of the range space of $C_{\varphi}$ and establishes a norm on these range spaces relative to which $C_\varphi $ acts isometrically. In particular, a basis for
the range spaces is constructed which provides the best approximation of the image maps.

\begin{thm}
Let $\varphi :\mathbb{C}\rightarrow \mathbb{C}$ be an entire function, whose restriction to the real line, $\varphi_{\mathbb{R}}$ maps $\mathbb{R}$ into $\mathbb{R}$. Suppose there exists a constant $c>0$ such that $m\varphi_{\mathbb{R}} ^{-1}(E) \leq c m(E)$ for every measurable set $E$, where $m$ is the Lebesgue measure on the Borel subsets of $\mathbb{R}$. In addition, suppose $m\; \ll \; m\circ \varphi_{\mathbb{R}}^{-1}$, i.e., $m\; \sim \; m\circ \varphi_{\mathbb{R}}^{-1}$ ($m$ and $m\circ \varphi_{\mathbb{R}}^{-1}$ are mutually absolutely continuous). Then the following are true:
\begin{enumerate}
\item[(a)] $C_{\phi}$ is a bijection from $B^{2}_{a}$ onto $ran(C_\varphi)$.
\item[(b)] $ran(C_{\varphi})$ is a reproducing kernel Hilbert space with norm defined by
\begin{center}
$\|C_{\varphi}f\|_{ran(C_{\varphi})}\; = \; \|f\|_{B^{2}_{a}}$, \quad $\forall f \in B^{2}_{a}$.
\end{center}
That is $C_{\varphi}$ is an isometry.
\end{enumerate}
\end{thm}
\begin{proof}
\begin{enumerate}
\item[(a)] The first part of this proof is due to Singh and Manhas (See ~\cite{SiMa93}).\\
Since $m\circ \varphi_{\mathbb{R}}^{-1} \; \ll \; m$, the Radon-Nikodym derivative $f_{\varphi_{\mathbb{R}}}$ of $m\circ \varphi_{\mathbb{R}}^{-1}$ with respect to $m$ exists and we have
\begin{center}
$\ds m\varphi_{\mathbb{R}}^{-1}(E)=\int_{E}f_{\varphi_{\mathbb{R}}}dm$, \quad for every Borel set $E \subseteq \mathbb{R}$.
\end{center}
Now suppose there exists one such $E$ with $m(E) > 0$ such that $f_{\varphi_{\mathbb{R}}}=0$ on $E$. Then from the above equality we have $m\varphi_{\mathbb{R}}^{-1}(E)=0$, which is a contradiction to the fact that $m\; \ll \; m\circ \varphi_{\mathbb{R}}^{-1}$. Hence $f_{\varphi_{\mathbb{R}}}$ is different from zero almost everywhere.\\
Then the corresponding multiplication operator $M_{f_{\varphi_{\mathbb{R}}}}$ is an injection.\\
But we also know for $f,g\in L^{2}(\mathbb{R})$, 
\begin{eqnarray}
 \ds \left(C_{\varphi_{\mathbb{R}}}^{\ast}C_{\varphi_{\mathbb{R}}}f\;,\;g\right)\; &= &\;\left(C_{\varphi_{\mathbb{R}}}f\;,\;C_{\varphi_{\mathbb{R}}}g\right) \nonumber \\
 &=& \; \int f\; \overline{g}\;dm \varphi_{\mathbb{R}}^{-1} \nonumber \\
 &=& \int f_{\varphi_{\mathbb{R}}} f\; \overline{g}\;dm  \nonumber \\
 &=& \left(M_{f_{\varphi_{\mathbb{R}}}}f\;,\;g\right). \nonumber
 \end{eqnarray}
Hence $C_{\varphi_{\mathbb{R}}}^{\ast}C_{\varphi_{\mathbb{R}}} \;=\; M_{f_{\varphi_{\mathbb{R}}}}$ is an injection on $L^{2}(\mathbb{R})$ and this implies $C_{\varphi_{\mathbb{R}}}$ is an injection on $L^{2}(\mathbb{R})$. Since $B^{2}_a \subset L^{2}(\mathbb{R})$, $C_{\varphi_{\mathbb{R}}}$ is an injection on $B^{2}_{a}$. Now we claim that $C_{\varphi}$ is an injection on $B^{2}_{a}$.\\
\textit{Proof of claim}: Let $C_{\varphi}f\;=\;0$ for some $f\in B^{2}_{a}$. This implies,
\begin{eqnarray}
f(\varphi(z)) &=& 0\quad \forall z\in \mathbb{C},\nonumber\\
f(\varphi(x)) &=& 0\quad \forall x\in \mathbb{R},\nonumber\\
f(\varphi_{\mathbb{R}}(x)) &=& 0\quad \forall x\in \mathbb{R},\nonumber\\
C_{\varphi_{\mathbb{R}}}f(x) &=& 0\quad \forall x\in \mathbb{R},\nonumber\\
C_{\varphi_{\mathbb{R}}}f &=& 0\quad \text{on} \quad \mathbb{R},\nonumber\\
f &=& 0\quad \text{on} \quad \mathbb{R}.\nonumber
\end{eqnarray}
Therefore by the identity theorem, $f=0\quad \text{on} \quad \mathbb{C}$. Hence $C_{\phi}$ is an injection from $B^{2}_{a}$ onto $ran(C_\varphi)$.

\item[(b)] We know $\|F\|_{ran(C_{\varphi})} = \inf \{\|f\|_{B^{2}_{a}}:\; C_{\varphi}f=F\}$. This implies,
\begin{eqnarray}
\|F\|_{ran(C_{\varphi})} &=& \inf\{{\|f-g\|_{B^{2}_{a}}: g\in\mathcal{N}(C_{\varphi}), C_{\varphi}f=F}\}\nonumber\\
&=& \|\mathcal{P}_{\mathcal{N}(C_{\varphi})^{\bot}}f\|_{B^{2}_{a}}\quad (\because \mathcal{N}(C_{\varphi})\; \text{is closed})\nonumber\\
&=& \|f\|_{B^{2}_{a}}\quad (\because \mathcal{N}(C_{\varphi}) = {0})\nonumber
\end{eqnarray}
Here, $\mathcal{N}(C_{\varphi})$ is the null space of $C_{\varphi}$ and $\mathcal{P}_{\mathcal{N}(C_{\varphi})^{\bot}}$ is the orthogonal projection from $B^{2}_{a}$ onto the orthogonal complement of $\mathcal{N}(C_{\varphi})$ in $B^{2}_{a}$. Hence $\|C_{\varphi}f\|_{ran(C_{\varphi})} \;=\; \|f\|_{B^{2}_{a}}$. This completes the proof.
\end{enumerate}
\end{proof}
\section{Bounded composition operators on de Branges-Rovnyak Spaces}

Let $\Pi^{+}\subset \mathbb{C}$ be the upper half of the complex plane. Let $g(z)$ be an entire function satisfying $|g(\overline{z})|<|g(z)|$, $\forall z \in \Pi^{+}$. The \textit{de Branges-Rovnyak space} $H(g)$ is the space of all entire functions $f(z)$ satisfying the following conditions:
\begin{center}
    \begin{enumerate}
        \item $\|f\|^{2}_{H(g)}=\ds \int_{-\infty}^{+\infty}\left|\frac{f(t)}{g(t)}\right|^{2}dt< \infty$,
        \item Both ratios $\frac{f(z)}{g(z)}$ and $\frac{\overline{f(\overline{z})}}{g(z)}$ are of bounded type and of nonpositive mean type in $\Pi^+$.
    \end{enumerate}
\end{center}
$H(g)$ is a Hilbert space with the norm defined above. It is well-known that $H(g)$ is a reproducing kernel Hilbert space with reproducing kernel given by $k_{w}(z)= \frac{i}{2}\frac{g(z)\overline{g(w)}-\overline{g(\overline{z})}g(\overline{w})}{\pi (z-\overline{w})}$.\\
Note that the Paley-Wiener space $B^{2}_{a}$ is the space $H(g)$ with $g(z)=e^{-iaz}$.

In the following two cases we characterize the bounded composition operators on $H(g)$ by imposing different conditions on the function $g$.

\noindent{\underline{\bf Case I}: When $g$ is of exponential type.}

\begin{lem}\label{JaSalem1}
Suppose $g$ is an entire function of exponential type $\sigma$ satisfying $|g(\overline{z})|<|g(z)|$, $\forall z \in \Pi^{+}$. Then an analytic function $f$ is of exponential type less than or equal to $\sigma$ if and only if $\frac{f(z)}{g(z)}$ and $\frac{\overline{f(\overline{z})}}{g(z)}$ are of nonpositive mean type in $\Pi^+$.
\end{lem}
\begin{proof}
Suppose $f$ is of exponential type less than or equal to $\sigma$. Then $f$ is of mean type in upper and lower half plane, say, $\sigma_+$ and $\sigma_-$, respectively, where $\sigma_+, \sigma_- \leq \sigma$. Now the mean type of $\frac{f(z)}{g(z)}$ in $\Pi^+$ is
\begin{eqnarray}
\ds \lim_{r\rightarrow \infty} \frac{2}{\pi r} \int_0^\pi \log \left|\frac{f(re^{i\theta})}{g(re^{i\theta})}\right|\sin \theta d\theta & = & \sigma_+ - \sigma \qquad \left(\text{since,}\quad |g(\overline{z})| < |g(z)| \quad \text{in} \quad \Pi^+\right)\nonumber\\
\ds  & \leq & 0 .\nonumber
\end{eqnarray}
Similarly, the mean type of $\frac{\overline{f(\overline{z})}}{g(z)}$ in $\Pi^+$ is
\begin{eqnarray}
\ds \lim_{r\rightarrow \infty} \frac{2}{\pi r} \int_0^\pi \log \left|\frac{f(re^{-i\theta})}{g(re^{i\theta})}\right|\sin \theta d\theta & = & \sigma_- - \sigma \qquad \left(\text{again since,}\quad |g(\overline{z})| < |g(z)| \quad \text{in} \quad \Pi^+\right)\nonumber\\
\ds  & \leq & 0 .\nonumber
\end{eqnarray}

Conversely, suppose both $\frac{f(z)}{g(z)}$ and $\frac{\overline{f(\overline{z})}}{g(z)}$ are of nonpositive mean type in $\Pi^+$. We need to show that $f$ is of exponential type less than or equal to $\sigma$. But, we know that, $\ds \limsup_{|z|\rightarrow \infty} \frac{\log |g(z)|}{|z|} = \sigma$. Since, $\frac{f}{g}$ is nonpositive mean type in $\Pi^+$, we have the following
\begin{eqnarray}
\ds \limsup_{y\rightarrow \infty} \frac{\log |\frac{f(iy)}{g(iy)}|}{y} & \leq & 0,\nonumber\\
\ds \limsup_{y\rightarrow \infty} \left(\frac{\log |f(iy)|}{y} - \frac{\log |g(iy)|}{y}\right) & \leq & 0,\nonumber\\
\ds \limsup_{y\rightarrow \infty} \frac{\log |f(iy)|}{y} & \leq & \liminf_{y\rightarrow \infty} \frac{\log |g(iy)|}{y}\nonumber\\
\ds  & \leq & \limsup_{y\rightarrow \infty} \frac{\log |g(iy)|}{y}\nonumber\\
\ds  & \leq & \sigma .\nonumber
\end{eqnarray}
This implies that $f$ is of mean type $\leq \sigma$ on $\Pi^+$. Also we know $\frac{\overline{f(\overline{z})}}{g(z)}$ is nonpositive mean type in $\Pi^+$, therefore
\begin{eqnarray}
\ds \limsup_{y\rightarrow \infty} \frac{\log |\frac{f(-iy)}{g(iy)}|}{y} & \leq & 0,\nonumber\\
\ds \limsup_{y\rightarrow \infty} \frac{\log |f(-iy)|}{y} & \leq & \liminf_{y\rightarrow \infty} \frac{\log |g(iy)|}{y}\nonumber\\
\ds & \leq & \limsup_{y\rightarrow \infty} \frac{\log |g(iy)|}{y}\nonumber\\
\ds & \leq & \sigma .\nonumber
\end{eqnarray}
This implies that $f$ is of mean type $\leq \sigma$ on $\Pi^-$.

Since, $f$ is entire, $f$ is of exponential type $\sigma_1$, then by Kre\v{\i}n's theorem $\sigma_1 = \max\{$mean type of $f$ on $\Pi^+$, mean type of $f$ on $\Pi^-$\} $\leq \sigma$.
\end{proof}

If $g$ is as in Lemma \ref{JaSalem1}, then every $f\in H(g)$ is of exponential type less than or equal to $\sigma$.
The proof of the following theorem on bounded composition operators on $H(g)$ is very similar to the proof of Theorem \ref{Chacon}.
\begin{thm}\label{JaSa3}
Suppose $g$ is an entire function of exponential type $\sigma$ satisfying $|g(\overline{z})|<|g(z)|$, $\forall z \in \Pi^{+}$. Let $\varphi: \mathbb{C} \rightarrow \mathbb{C}$ be a non-constant holomorphic map. If the composition operator $C_\varphi$ sends $H(g)$ into itself, then $\varphi$ is affine.
\end{thm}
\begin{proof}
Suppose that $C_\varphi : H(g) \rightarrow H(g)$. Since both $k_w\circ \varphi$  and $k_w$ are in $H(g)$, by Lemma \ref{JaSalem1} both are exponential type less than or equal to $\sigma$ and therefore of order 1. Then by Polya's theorem (\cite{Po26}) $\varphi$ is a polynomial. Let degree of $\varphi$ be $n$. We will show that $n = 1$.

Let $\ds \varphi(z)=\sum_{k=0}^{n}c_k z^k$, with $|c_n| \neq 0$. Then by Cauchy's integral formula we can show that $M_\varphi (r) \geq |c_n| r^n$ for large r. Now, without loss of generality we may assume that $\varphi(0) = 0$. Then there exists a constant $c \in (0,1)$, such that
\begin{center}
$M_{f\circ \varphi}(r) \geq M_f \left(cM_\varphi \left(\frac{r}{2}\right)\right) \geq M_f \left(c|c_n|\frac{r^n}{2^n}\right)$,
\end{center}
for each function $f$ in $H(g)$.

Now, suppose $f \in H(g)$. Since the order of $f$ is one, for arbitrarily large value of $r$ the inequality $M_f(r) \geq \exp{r^b}$ holds when $0 < b < 1$. Hence, from the above inequality, there are arbitrarily large values of $r$ such that
\begin{center}
$M_{f\circ \varphi}(r) \geq \exp{\left(\left(c|c_n|\frac{r^n}{2^n}\right)^b\right)} = \exp{\left(c^b|c_n|^b\frac{r^{nb}}{2^{nb}}\right)}$.
\end{center}
Since the order of $f\circ \varphi$ is one, there exist constants $A, B$ such that
\begin{center}
$M_{f\circ \varphi}(r) \leq A\exp{(Br)}$, \quad for all $r$.
\end{center}
Thus there are arbitrarily large values of $r$ such that $\exp{\left(c^b|c_n|^b\frac{r^{nb}}{2^{nb}}\right)} \leq A\exp{(Br)}$. This implies $nb \leq 1$, but $b$ is any number less than one, hence $n$ must be one (since $\varphi$ is not a constant function).
\end{proof}
\begin{thm}\label{JaSa4}
Suppose $g$ is an entire function of exponential type $\sigma$ satisfying $|g(\overline{z})|<|g(z)|$, $\forall z \in \Pi^{+}$. Let $\varphi(z) = az+b$, $z\in \mathbb{C}$ with $0<|a|\leq 1$, and $a\in \mathbb{R}$. Then the operator $C_\varphi$ is bounded on $H(g)$ if $|\frac{g(at+b)}{g(t)}|\leq c$, for some constant $c$ and for all $t\in \mathbb{R}$.
\end{thm}
\begin{proof}
Let $f\in H(g)$. Then the order and type of $f$ are respectively $1$ and $\sigma_1$, where $0 < \sigma_1 \leq \sigma$.

Denote $F = f \circ \varphi$.
Then, as in the proof of Theorem (\ref{Chacon-cor-weight}) the order and type of $f(az)$ are $1$ and $|a|\sigma_1$ respectively. Also by a simple calculation we can show that the order and type are invariant under translation. Hence the order and type of $F$ are $1$ and $|a|\sigma_1 (\leq \sigma)$ respectively. Hence by Lemma \ref{JaSalem1}, $\frac{F}{g}$, and $\frac{F^*}{g}$ are of nonpositive mean type in $\Pi^+$, where $F^*(z) = \overline{F(\overline{z})}$.
Since both $F$ and $g$ are of exponential type less than or equal to $\tau$, for $\tau \geq \sigma$, it is clear from the fact $\frac{F(z)}{g(z)} = \frac{F(z)/e^{-i\tau z}}{g(z)/e^{-i\tau z}}$ that $\frac{F}{g}$ is bounded type in $\Pi^+$. Similarly, we can show that $\frac{F^*}{g}$ is also bounded type in $\Pi^+$.

Now assuming $b=p+iq$, we have the following,
\begin{eqnarray}
\ds \int_{\mathbb{R}} \left|\frac{(f\circ \varphi)(t)}{g(t)}\right|^2 dt & = & \int_{\mathbb{R}} \left|\frac{f(at+b)}{g(t)}\right|^2 dt\nonumber\\
\ds & \leq & c^2\int_{\mathbb{R}} \left|\frac{f(at+b)}{g(at+b)}\right|^2 dt\nonumber\\
\ds & \leq & \frac{c^2}{|a|}\int_{\mathbb{R}} \left|\frac{f(x+p+iq)}{g(x+p+iq)}\right|^2 dx\nonumber\\
\ds & \leq & \frac{c^2}{|a|}\int_{\mathbb{R}} \left|\frac{f(x)}{g(x)}\right|^2 dx,\nonumber
\end{eqnarray}
which is finite. Here the last inequality follows from Plancherel-P\'{o}lya theorem (see ~\cite{Le96}, section 7.4) and this completes the proof.
\end{proof}

\noindent{\underline{\bf Case II}: When $(g)^{-1}\in L^2(\mathbb{R})$.}

The following theorem  gives a sufficient condition for composition operators to be bounded on the de Branges-Rovnyak space $H(g)$ when $\frac{1}{g} \in L^2(\mathbb{R})$.

\begin{thm}\label{JaSa2}
Suppose $g$ is an entire function satisfying $|g(\overline{z})|<|g(z)|$, $\forall z \in \Pi^{+}$ and also $(g)^{-1}\in L^{2}(\mathbb{R})$. Define, $d\lambda = \frac{1}{|g(t)|^{2}}dt$. Let $\varphi$ be an entire function such that its restriction to the real line, $\varphi_{\mathbb{R}}$, maps $\mathbb{R}$ into $\mathbb{R}$. If $C_{\varphi}$ maps $H(g)$ into $H(g)$ then there exists a positive constant $c$ such that $\lambda \varphi_{\mathbb{R}}^{-1}(E) \leq c \lambda(E)$, for every measurable set $E \subseteq \mathbb{R}$.
\end{thm}
\begin{proof}
Suppose $f\in H(g)$. Then $f\in L^{2}(\lambda)$. Since $C_{\varphi}$ sends $H(g)$ into $H(g)$, $f\circ \varphi \in H(g)$; that is, $f\circ \varphi \in L^{2}(\lambda)$ and this implies $f\circ \varphi_{\mathbb{R}} \in L^{2}(\lambda)$. Then since $\lambda$ is integrable, it is a $\sigma$-finite measure on $\mathbb{R}$, hence by Theorem \ref{Singh} there exists a positive constant $c$ such that $\lambda \varphi_{\mathbb{R}}^{-1}(E) \leq c \lambda(E)$, for every measurable set $E \subseteq \mathbb{R}$.
\end{proof}

\bibliographystyle{amsalpha}

\end{document}